\documentclass[11pt,final]{amsart}
\usepackage[active]{srcltx}
\usepackage[latin1]{inputenc}
\usepackage{amsmath}
\usepackage{amsfonts}
\usepackage{amssymb}
\usepackage[colorlinks=true,urlcolor=blue,
citecolor=red,linkcolor=blue,linktocpage,pdfpagelabels,bookmarksnumbered,bookmarksopen]{hyperref}
\usepackage{graphicx, tikz}

\tikzset { domaine/.style 2 args={domain=#1:#2} }
\tikzset{
xmin/.store in=\xmin, xmin/.default=-3, xmin=-3,
xmax/.store in=\xmax, xmax/.default=3, xmax=3,
ymin/.store in=\ymin, ymin/.default=-3, ymin=-3,
ymax/.store in=\ymax, ymax/.default=3, ymax=3,
}

\usepackage[english]{babel}
\usepackage{marginnote}
\usepackage[left=2.95cm,right=2.95cm,top=2.8cm,bottom=2.8cm]{geometry}
\numberwithin{equation}{section}
\usepackage{mathrsfs}
\usepackage{color}
\usepackage{amsmath, bbm}
\usepackage{amsfonts}
\usepackage{amssymb}
\usepackage{graphicx}%
\providecommand{\U}[1]{\protect\rule{.1in}{.1in}}
\definecolor{linkcolor}{rgb}{0.00,0.50,0.00}
\providecommand{\U}[1]{\protect\rule{.1in}{.1in}}
\newtheorem{theorem}{Theorem}[section]
\newtheorem{proposition}[theorem]{Proposition}

\newtheorem{remark}{Remark}[section]

\numberwithin{equation}{section}

\newcommand{\pical}{\mathcal{P}}
\newcommand{\dd}{\mathrm{d}}

\newcommand{\res}{\mathop{\hbox{\vrule height 7pt width .5pt depth 0pt \vrule height .5pt width 6pt depth 0pt}}\nolimits}
\newcommand{\deb}{\rightharpoonup}

\newcommand{\haus}{\mathcal H}

\newcommand{\Wc}{\mathcal{T}}
\newcommand{\lcal}{\mathcal L}

\newcommand{\E}{\mathcal E}

\newcommand{\ve}{\varepsilon}

\newcommand{\R}{\mathbb R}

\newcommand{\Pu}{\mathcal{P}_1(\R^d)}
\newcommand{\Pd}{\mathcal{P}_2(\R^d)}

\DeclareMathOperator{\spt}{spt}

\def\spt{{\rm{spt}}} 
\def\dd{{\rm d}} 

\usepackage{babel}
\reversemarginpar
\begin{document}
\title[Dealing with moment measures via entropy and optimal transport]
{Dealing with moment measures\\ via entropy and optimal transport}
\author[F. Santambrogio]{Filippo Santambrogio}
\address{Laboratoire de Math\'ematiques d'Orsay, Universit\'e Paris-Sud, 91405 Orsay Cedex, France}
\email{filippo.santambrogio@math.u-psud.fr}
\maketitle

\begin{abstract} A recent paper by Cordero-Erausquin and Klartag provides a characterization of the measures $\mu$ on $\R^d$ which can be expressed as the moment measures of suitable convex functions $u$, i.e. are of the form $(\nabla u)_\#e^{- u}$ for $u:\R^d\to\R\cup\{+\infty\}$ and finds the corresponding $u$ by a variational method in the class of convex functions. Here we propose a purely optimal-transport-based method to retrieve the same result. The variational problem becomes the minimization of an entropy and a transport cost among densities $\rho$ and the optimizer $\rho$ turns out to be $e^{-u}$. This requires to develop some estimates and some semicontinuity results for the corresponding functionals which are natural in optimal transport. The notion of displacement convexity plays a crucial role in the characterization and uniqueness of the minimizers.
\end{abstract}
\section{Introduction}

We consider in this paper the notion of {\it moment measure} of a convex function, which comes from functional analysis and convex geometry. Given a convex function $ u:\R^d\to\R\cup\{+\infty\}$, we define its moment measure as 
$$\mu:=(\nabla u)_\#\rho,\qquad\mbox{where }\dd\rho=e^{- u(x)}\dd x.$$

The connection of this notion with the theory of optimal transport is straightforward from the fact that, by Brenier's Theorem, the map $\nabla u$ will be the optimal transport map for the quadratic cost $c(x,y)=\frac 12|x-y|^2$ from $\rho$ to $\mu$. 

In a recent paper, Cordero-Erausquin and Klartag (\cite{CorKla}) studied the conditions for a measure $\mu$ to be the moment measure of a convex function. First, they identified that an extra requirement has to be imposed to the function $ u$ in order the problem to be meaningful. The main difficulty arises in case $ u$ is infinite out of a proper convex set $K\subset\R^d$. In this case one needs to require some continuity properties of $ u$ on $\partial K$. Without this condition, every measure with finite first moment can be the moment of a function $ u$, which is in general discontinuous on $\partial\{ u<+\infty\}$. Also, there is a strong non-uniqueness of $ u$. On the contrary, if one restricts to convex functions $ u$ that are continuous $\haus^{d-1}-$a.e. on $\partial\{ u<+\infty\}$ (those functions are called {\it essentially continuous}, then there is a clear characterization: a measure $\mu$ is a moment measure if and only if it has finite first moment, its barycenter is $0$, and it is not supported on a hyperplane. Moreover, the function $ u$ is uniquely determined by $\mu$ up to space translations.

In \cite{CorKla}, the authors first prove that these conditions on $\mu$ are necessary, due to summability properties of log-concave densities, and they prove that they are sufficient to build $ u$ as the solution of a certain minimization problem.

Here we want to reprove the same existence result with a different method, replacing functional inequalities techniques with ideas from optimal transport. This aspect seems to be absent from \cite{CorKla} even if it is not difficult to translate most of the ideas and techniques of Cordero-Erasquin and Klartag into their optimal transport counterparts. The result is an alternative language, that is likely to be appreciated by people knowing optimal transport theory, while the community of functional inequalities could legitimately prefer the original one. The question of which approach will the colleagues working with both optimal transport and functional inequality prefer is an open and unpredictable issue\dots

The main idea justifying this approach is the following: many variational problems of the form 
$$\min \left\{\frac 12W_2^2(\rho,\mu)+ \int f(\rho(x))\,\dd x\right\}$$
have been studied in recent years, for different purposes (time-discretization of gradient flows, urba planning\dots\ see for instance \cite{buttazzo,demesave,T+GV} and Chapter 7 in \cite{OTAM}). The optimality condition of this problem reads, roughly speaking, as 
$$\phi+f'(\rho)=const,$$
where $\phi$ is the Kantorovich potential in transport from $\rho$ to $\mu$ for the cost $c(x,y)=\frac 12|x-y|^2$ (notice that the convex function appearing in Brenier's Theorem is given by $ u(x)=\frac 12 x^2-\phi(x)$: if $T=\nabla  u$ is the optimal transport, $-\nabla\phi$ is the optimal dispacement, i.e. $T(x)=x-\nabla\phi(x)$).

In the case $f(t)=t\ln t$, the condition above implies that $\rho$ is proportional to $e^{-\phi}$. In order to obtain the condition which is required we need to correct the above minimization, so that we get $ u$ instead of $\phi$. In order to do so, one needs to change the sign and insert a $\frac 12|x|^2$ term, which leads to
$$\min \left\{-\frac 12W_2^2(\rho,\mu)+ \frac 12\int |x|^2\,\dd\rho(x)+\E(\rho)\right\},$$
where $\E(\rho)$ denotes the entropy of $\rho$, defined as $\E(\rho):=\int\rho\ln\rho\,\dd x$ for $\rho\ll\lcal^d$, $\E(\rho)=+\infty$ for $\rho$ non absolutely continuous.
When $\mu$ has finite second moment and we minimize the above functional among measures $\rho$ supported on a given compact set $K$ it is easy to check that a minimizer exist and that we have $\rho=e^{- u}$, where $(\nabla u)_\#\rho=\mu$. Yet, it is highly possible that $ u$ is discontinuous at the boundary of $K$. Indeed, at least in the case where also $\mu$ is compactly supported, one sees that the function $ u$ must be Lipschitz (since its gradient is bounded), which means that it is bounded, and hence $\rho$ is bounded from below by a positive constant. Yet, $\rho=0$ and $ u=+\infty$ outside $K$ and $ u$ is not essentially continuous. This is a confirmation that every $\mu\in\Pu$ is a moment measure, if we accept convex functions $ u$ which are not essentially continous. 

The interesting case is the one where we minimize among measures $\rho\in\Pu$, without restricting their support to a compact domain. In this case the existence of an optimal $\rho$ is not evident (indeed, the term $W_2(\rho,\mu)$ is continuous for the weak convergence of probability measures on compact sets, but only l.s.c. on unbounded sets, and here it is accompanied by the negative sign, which makes it u.s.c., while we want to minimize). Also, the lower semicontinuity of the entropy term $\mathcal E(\rho)$ is more delicate on unbounded sets. Here comes into play the assumptions on $\mu$, which will allow to provide a bound on the first moment of a minimizing sequence $\rho_n$. Then, we can prove that $\rho$ is log-concave and that a precise representative of $\rho$ must vanish $\haus^{d-1}$-a.e. on the boundary of its support, which proves that $ u$ is essentially continous. In order to handle the case of measures $\rho,\mu\notin\Pd$, we will profit of the fact that the second moment part of $\frac 12W_2(\rho,\mu)$ and $ \frac 12\int |x|^2\,\dd\rho(x)$ cancel each other, and that everything is well-defined for $\rho,\mu\in \Pu$ if we transform the minimal transport problem for the cost $c(x,y)=\frac 12|x-y|^2$ into the maximal transport problem for $c(x,y)=x\cdot y$.

Hence, we study the minimization problem 
$$(P)\qquad\min\left\{\E(\rho)+\Wc(\rho,\mu)\;:\,\rho\in\Pu\right\},$$
where 
$$\Wc(\rho,\mu)=\sup\left\{\int(x\cdot y)\,\dd\gamma(x,y)\;:\;\gamma\in\Pi(\mu,\nu)\right\}.$$

As we announced above, we prove that this problem has a log-concave solution of the form $\rho=e^{- u}$, that $ u$ is essentially continuous and $\mu$ is the moment measure of $ u$. We also prove uniqueness up to translations of the minimizer and that the condition $\mu=(\nabla u)_\#\rho$, where $\dd\rho=e^{- u}$, is sufficient to minimize. This characterizes $ u$. Uniqueness and sufficient conditions will be based on the notion of displacement convexity in the space $\Pd$ endowed with the distance $W_2$. It is useful to notice that the displacement convexity of the entropy exactly corresponds to the Prekopa inequality, which expresses more or less the same fact, but at the level of convex functions.

\noindent {\bf Structure of the paper} After this introduction, Section 2 presents the main well-known tools from optimal transport theory (Wasserstein distances, geodesic interpolation\dots). The esitmates and semicontinuity of the entropy term that we provide at the end of the section are also well-known, but presented in alternative fashion.
Section 3 is devoted to the transport cost $\rho\mapsto \Wc(\rho,\mu)$ and to its estimates and semicontinuity in a $\Pu$ framework.
Section 4 is the core of the paper, and presents the variational problem that we need to solve in order to find the log-concave measure $\rho=e^{-u}$ that we aim at. In particular, we prove that $u$ is convex and essentially continuous.
In Section 5 we show that the condition $(\nabla u)_\#e^{-u}=\mu$ is actually equivalent to the fact that $e^{-u}$ minimizes our functional, using the notion of displacement convexity.
Finally, in Section 6 we compare our approach to the one in \cite{CorKla}, explaining why they are equivalent and how to pass from one to the other.

\noindent {\bf Acknowledgments} The authors would like to thank Guillaume Carlier and Dario Cordero-Erausquin for interesting discussions about this problem. These discussions have been made possible by the workshop ``New Trends in Optimal Transport'' organized at the Hausdorff Center for Mathematics in March 2015. 

\section{Few words on optimal transport, entropy and technical tools}

We recall here the main notions and notations that we will use throughout the paper. We refer to \cite{OTAM} (Chapters 1, 5 and 7) and to \cite{AmbGigSav,villani} for more details and complete proofs.

Given two probability measures $\mu,\nu\in \pical(\R^d)$ we consider the set of transport plans
$$\Pi(\mu,\nu)=\{\gamma\in\pical(\R^d\times \R^d):\,(\pi_x)_{\#}\gamma=\mu,\,(\pi_y)_{\#}\gamma=\nu,\}$$
i.e. those probability measures on the product space having $\mu$ and $\nu$ as marginal measures.

For a cost function $c:\R^d\times\R^d\to [0,+\infty]$ we consider the minimization problem
$$ \min\left\{\int c\,\dd\gamma\;:\:\gamma\in\Pi(\mu,\nu)\right\},$$
which is called the Kantorvich optimal transport problem for the cost $c$ from $\mu$ to $\nu$. In particular, we consider the case $c(x,y)=\frac 12|x-y|^2$. In this case the above minimal value is finite whenever $\mu,\nu\in\Pd$, where $\Pd:=\left\{\rho\in\pical(\R^d)\,:\,\int |x|^2\,\dd\rho(x)<+\infty\right\}$. 

For the above problem one can prove that the minimal value also equals the maximal value of a dual problem
$$\max\left\{ \int\phi \,\dd\mu+\int\psi\,\dd\nu\;:\;\phi(x)+\psi(y)\leq \frac12|x-y|^2\right\},$$
and that the optimal function $\phi$ may be used to construct an optimizer $\gamma$. Indeed, the optimal $\phi$ is locally lipschitz and semiconcave (in particular $x\mapsto \frac 12 |x|^2-\phi(x)$ is convex) on $\spt(\mu)$ and differentiable $\mu-$a.e. if $\mu\ll\lcal^d$; one can define a map $T:\R^d\to\R^d$ through $T(x)=x-\nabla\phi(x)$ and this map satisfies $T_\#\mu=\nu$ and $\gamma_T:=(id,T)_\#\mu$ (i.e. the image measure of $\mu$ through the map $x\mapsto (x,T(x))$) belongs to $\Pi(\mu,\nu)$ and is optimal in the above problem. Moreover, the map $T$ is the gradient of the convex function $u$ given by $u(x)=\frac 12 |x|^2-\phi(x)$ and is called the optimal transport map (for the quadratic cost $c(x,y)=\frac 12|x-y|^2$) from $\mu$ to $\nu$. The fact that the optimal transport map $T$ exists, is unique, and is the gradient of a convex function is known as Brenier Theorem (see \cite{Brenier polar}).

The same could be obtained if one withdrew from the cost $\frac 12|x-y|^2$ the parts $\frac12|x|^2$ and $\frac12|y|^2$ which only depend on one variable each (hence, their integral against $\gamma$ only depends on its marginals). In this case we are interested in a transport maximization problem
$$ \max\left\{\int (x\cdot y)\,\dd\gamma(x,y)\;:\:\gamma\in\Pi(\mu,\nu)\right\}$$
and the dual problem would become
$$\min\left\{ \int u \,\dd\mu+\int v\,\dd\nu\;:\;u(x)+v(y)\geq x\cdot y\right\}.$$
In this problem it is quite clear that any pair $(u,v)$ can be replaced with $(u,u^*)$ where $u^*(y):=\sup_x x\cdot y-u(x)$ is the Legendre transform of $u$ and is the smalles function compatible with $u$ in the constraint $u(x)+v(y)\geq x\cdot y$. Then, it is easy to see by the primal-dual optimality conditions that the optimal $\gamma$ and the optimal $u$ satisfy
$$ \spt(\gamma)\subset\{(x,y)\,:\,u(x)+u^*(y)=x\cdot y\}=\{(x,y)\,:\,y\in\partial u(x)\},$$
which shows that $\gamma$ is concentrated on the graph of a map $T$ (which is one-valued $\mu-$a.e., provided $\mu\ll\lcal^d$), given by $T=\nabla u$.
 
The value of the minimization problem with the quadratic cost may also be used to define a quantity, called Wasserstein distance, over $\Pd$:
$$W_2(\mu,\nu):=\sqrt{ \min\left\{\int |x-y|^2\,\dd\gamma\;:\:\gamma\in\Pi(\mu,\nu)\right\}}.$$
This quantity may be proven to be a distance over $\Pd$. Moreover, when restricted to the probabilities supported on a given compact set, i.e. to $\pical(K)$ with $K\subset \R^d$ compact, it metrizes the weak-* convergence of probability measures. The space $\Pd$ endowed with the distance $W_2$ is called Wasserstein space of order $2$ and denoted in this paper by $\mathbb W_2(\R^d)$. 

The geodesics in this space play an important role in the theory of optimal transport. If $\mu,\nu\in\Pd$ and $\mu\ll\lcal^d$, we define $\rho_t:=((1-t)id+tT)_\#\mu$, where $T$ is the optimal trnasport from $\mu$ to $\nu$. This curve $\rho_t$ happens to be a constant speed geodesic for the distance $W_2$ connecting $\mu$ to $\nu$ (in case neither $\mu$ nor $\nu$ are absolutely continuous, it is possible to produce a geodesic by taking $\rho_t:=(\pi_t)_\#\gamma$ where $\pi_t(x,y)=(1-t)x+t y$ and $\gamma$ is optimal in the Kantorovich problem, which gives the same result if $\gamma=\gamma_T$).

Once we know the geodesics in $\mathbb W_2(\R^d)$, one can wonder which functionals $F:\Pd\to\R$ are geodesically convex, i.e. convex along constant speed geodesics. This notion, applied to the case of the Wasserstein spaces, is also called {\it displacement convexity} and has been introduced by McCann in \cite{mccann}. It is very useful both to provide uniqueness results for variational problems and to provide sufficient optimality conditions. A related notion is that of {\it convexity on generalized geodesics}, which corresponds to $t\mapsto F(\rho_t)$ being convex every time that we take a triplet $\mu,\rho_0,\rho_1$, the optimal maps $T_0$ from $\mu$ to $\rho_0$ and $T_1$ from $\mu$ to $\rho_1$, and take $\rho_t=((1-t)T_0+tT_1)_\#\mu$. The curve $\rho_t$ is not in general the geodesic connecting $\rho_0$ to $\rho_1$. However, the conditions to guarantee displacement convexity and convexity on generalized geodesics are often very similar and most of the useful functionals which are actually used satisfy both notions, and we prenset this notion here only for the sake of completeness.

We are in particular interested in the entropy functional $\E$, defined as follows
$$\E(\rho):=\begin{cases}\int \rho(x)\ln(\rho(x))\,\dd x &\mbox{ if }\rho\ll\lcal^d,\\
					+\infty &\mbox{ otherwise.}\end{cases}$$
This functional is displacement convex as it satisfies the assumption required in \cite{mccann} (it is also convex on generalized geodesics, but we will not use it here). We also compute its derivative along a geodesic $\rho_t$. Suppose $\rho_t=(T_t)_\#\rho$, where $T_t=(1-t)id+tT$. From a simple change of variable, we have
$$\rho_t(T_t(x))=\frac{\rho(x)}{\det((1-t)I+tDT(x))}$$
(this formula, where $DT$ is the Jacobian matrix of $T$, is valid provided $T_t$ is countably Lipschitz and injective, which is the case whenever $T$ is an optimal transport). Hence we have
$$\E(\rho_t)=\int \ln(\rho_t(x))\,\dd\rho_t(x)=\int \ln\rho_t(T_t(x))\,\dd\rho(x)=\E(\rho)-\int \ln(\det((1-t)I+tDT(x)))\,\dd\rho(x).$$
Hence, we have
$$\frac{d}{dt}\left[\E(\rho_t)\right]_{|t=0}=-\int \nabla\cdot (T-id)\,\dd\rho.$$
Here the divergence is to be understood in the a.e. sense, as $T$ is countably Lipschitz. If we use $T=\nabla v$ with $v$ convex, this divergence is equal to $\Delta^{ac} v-d$, where $\Delta^{ac}$ is the absolutely continuous part of the distributional Laplacian of $v$, which is a measure.

We also need to underline two other properties of $\E$, in particular lower bounds and semicontinuity. We stress that both these conditions are easy when we look at measures in $\pical(K)$ with $|K|<\infty$, but become trickier on the whole space, because $\rho\ln\rho$ is not positive.
%

The last property that we need to recall is the semicontinuity of $\E$ w.r.t. the weak-* convergenc eof probability measures, under the extra condition of a bound on the first moment $\int|x|\,\dd \rho(x)$. The semicontinuity of functionals of the form 
$$\rho\mapsto \int f\left(\frac{\dd\rho}{\dd\lambda}(x)\right)\dd\lambda(x)$$
 is standard for convex and superlinear $f$ (which is the case for $f(t)=t\ln(t)$) whenever the reference measure $\lambda$ (which is the Lebesgue measure here) is finite (see for instance \cite{boubu}  or Chapter 7 in \cite{OTAM}). If $f$ is positive it it easy to get the same result by taking the supremum of functionals restricted to compact subsets, but this is not possible here. Yet, we can notice that for $\rho\ll\lcal^d$ we have
 $$\E(\rho)=\int \left(\rho(x)\ln(\rho(x))+e^{h(x)-1}-\rho(x)h(x)\right)\dd x+\int \rho(x)h(x)\,\dd x-\int e^{h(x)-1}\dd x$$
for any function $h$ such that $e^{h(x)-1}\in L^1(\R^d)$ and $h\in L^1(\rho)$. If we suppose $\rho\in \Pu$ we can take $h=-\sqrt{|x|}$. Then we can write $\E=\E_1+\E_2+\E_3$, where
\begin{eqnarray*}
\E_1(\rho)&=&\int \left(\rho(x)\ln(\rho(x))+e^{h(x)-1}-\rho(x)h(x)\right)\dd x\\
\E_2(\rho)&=&\int \rho(x)h(x)\,\dd x,\\
\E_3(\rho)&=&-\int e^{h(x)-1}\dd x.
\end{eqnarray*}
Notice that the integrand in $\E_1$ is positive (indeed, for every $a\in\R_+$ and $b\in\R$ we have $a\ln(a)+e^{b-1}\geq a\cdot b$). This allows to write $\E\geq \E_2+\E_3$, and hence we have
$$\E(\rho)\geq -\int\rho(x)\sqrt{|x|}\dd x-\int e^{-\sqrt{|x|}-1}\dd x\geq -C-\sqrt{\int|x|\,\dd\rho(x)},$$
where we used $\int \dd\rho(x)=1$ in the last inequality (H\"older inequality). This gives a bound from below of $\E(\rho)$ in terms of the square root of the first moment of $\rho$.

Then, we also express $\E_1$ as a supremum over compact sets, using the positivity of the integrand:
$$\E_1(\rho)=\sup_{K\subset \R^d,\,K\mbox{ compact}}\int \left(\rho(x)\ln(\rho(x))+e^{h(x)-1}-\rho(x)h(x)\right)\dd x.$$

Now, we observe that for every sequence $\rho_n\deb\rho$ with $\int|x|\dd\rho_n(x)\leq C$ and $\E(\rho_n)\leq C$ we also have weak convergence $\rho_n\deb\rho$ in $L^1$, and $\int \psi(x)\dd \rho_n(x)\to \int \psi(x)\dd \rho(x)$ for every sublinear function $\psi$ (i.e. satisfying $\lim_{|x|\to\infty}\psi(x)/|x|=0$, which is true for $\psi=h$). Hence, the term $\E_1$ is l.s.c., $\E_2$ is continuous and $\E_3$ is constant and finite for this type of convergence. Globally we can resume these facts in the following proposition.
\begin{proposition}\label{entropy}
\begin{enumerate}
\item The functional $\E:\Pu\to\R\cup\{+\infty\}$ is well-defined and satisfies $\E(\rho)\geq -C- \left(\int|x|\,\dd \rho(x)\right)^{1/2}$.
\item
For every sequence $\rho_n\deb\rho$ with $\int|x|\dd\rho_n(x)\leq C$ and $\E(\rho_n)\leq C$ we have $\liminf_n \E(\rho_n)\geq \E(\rho)$.
\item 
When restricted to $\Pd$, the functional $\E$ is geodesically convex in $\mathbb W_2$ and strictly convex on every geodesic $t\mapsto \rho_t=((1-t)id+tT)_\#\rho$ where the map $T$ is not a translation.
\item If $\rho_t=((1-t)id+tT)_\#\rho$, then the derivative at $t=0$ of $t\mapsto \E(\rho_t)$ is given by $-\int \nabla\cdot (T-id)\,\dd\rho$.
\end{enumerate}
\end{proposition}

\section{The maximal correlation functional}\label{2}

For $\mu,\rho \in \Pu$, with $\int y\,\dd\mu(y)=0$, we define the following quantity:
\begin{eqnarray*}
\Wc(\rho,\mu)&:=&\sup\left\{\int(x\cdot y)\,\dd\gamma\;:\;\gamma\in\Pi(\mu,\nu)\right\}\\
&=&\inf\left\{\int  u\,\dd\rho+\int  u^*\,\dd\mu\;:\; u:\R^d\to\R\cup\{+\infty\}\mbox{ convex and l.s.c.}\right\}.
\end{eqnarray*}

We notice that, if $\rho,\mu\in\Pd$, then we also have
$$\Wc(\rho,\mu)=\frac12\int |x|^2\,\dd\rho(x)+\frac12\int |y|^2\,\dd\mu(y)-\frac 12W_2^2(\rho,\mu).$$

This functional $\Wc$ is a transport cost, but we also observe that it is the maximal correlation between $\rho$ and $\mu$, in the sense
$$\Wc(\rho,\mu)=\sup\left\{ \mathbb E[X\cdot Y]\;:\;X\sim\rho, Y\sim\mu\right\}.$$
For this reason, $\Wc$ will be called {\it maximal correlation functional}.

We are interestend in the following properties.
\begin{proposition}\label{lscT}
\begin{enumerate}
\item For every $\rho\in\Pu$, we have $\Wc(\rho,\mu)\in[0,+\infty]$.
\item
If $\rho$ and $\tilde\rho$ are one obtained from the other by translation, then $\Wc(\rho,\mu)=\Wc(\tilde\rho,\mu)$.
\item
If $\int x\,\dd\rho_n(x)=0$ and $\rho_n\deb\rho$, then $\liminf_n \Wc(\rho_n,\mu)\geq \Wc(\rho,\mu)$.
\item For every $\mu\in\Pu$ with $\int y\,\dd\mu(y)=0$ there exists a sequence $\mu_n$ of compactly supported measures with $\mu_n\deb\mu$ and $\int y\,\dd\mu(y)=0$ such that for every $\rho\in\Pu$ we have $\Wc(\rho,\mu_n)\to\Wc(\rho,\mu)$. 
\item For every $\rho\in\Pu$ there exists a sequence $\rho_n$ of compactly supported probabilities such that $\rho_n\deb\rho$, $\E(\rho_n)\to\E(\rho)$ and such that for every $\mu\in\Pu$ with $\int y\,\dd\mu(y)=0$ we have $\Wc(\rho_n,\mu)\to\Wc(\rho,\mu)$. 
\end{enumerate}
\end{proposition}
\begin{proof}
In order to prove (1), just take $\gamma=\rho\otimes\mu\in\Pi(\rho,\mu)$. We get 
$$\Wc(\rho,\mu)\geq \int(x\cdot y)\,\dd(\rho\otimes\mu)=\left(\int x\,\dd\rho(x)\right)\cdot\left(\int y\,\dd\mu(y)\right)=0.$$
To prove (2), notice that every $\tilde\gamma\in\Pi(\tilde\rho,\mu)$ can be expressed as the translation of a $\gamma\in\Pi(\rho,\mu)$, in the sense $\int \phi(x,y)\,\dd\tilde\gamma(x,y)=\int\phi(x+v,y)\,\gamma(x,y)$, where $v$ is the vector translating $\rho$ into $\tilde\rho$. Applying this fact to $\phi(x,y)=x\cdot y$, we get $\Wc(\tilde\rho,\mu)=\Wc(\rho,\mu)+\int v\cdot y\,\dd\mu(y)=\Wc(\rho,\mu)$.

We now prove (3). We restrict to the case $\Wc(\rho_n,\mu)\leq C$ otherwise the statement is straightforward. We first take optimal plans $\gamma_n$ such that $\Wc(\rho_n,\mu)=\int(x\cdot y)\,\dd\gamma_n(x,y)$. From the tightness assumption on $\rho_n$, we infer that $\gamma_n$ are also tight. Hence, we can extract a converging subsequence $\gamma_n\deb\gamma\in\Pi(\rho,\mu)$. If we set $\Gamma_n:=\spt(\gamma_n)$, we have a sequence of closed sets in $\R^d\times\R^d$. We can extact a further subsequence locally Hausdorff converging to a closed set $\Gamma$. From the cyclical monotonicity of each $\Gamma_n$, we infer that $\Gamma$ is also cyclically monotone. Hence, $\gamma$ is also optimal, since its support is contained in $\Gamma$, which implies that it is cyclically monotone, a condition which is sufficient to guarantee optimality. Hence we have $\Wc(\rho,\mu)=\int(x\cdot y)\,\dd\gamma(x,y)$.

We just need to prove $\liminf_n \int(x\cdot y)\,\dd\gamma_n(x,y)\geq \int(x\cdot y)\,\dd\gamma(x,y)$. If $x\cdot y$ were a bounded continuous function, we would have equality. The problem is that it is not bounded. Yet, we can prove that it is bounded from below on $\bigcup_n \Gamma_n$, which is enough.

Indeed, take $(x,y),(x',y')\in\Gamma_n$. From cyclical monotonicity we can write
$$x\cdot y+x'\cdot y'\geq x\cdot y'+x'\cdot y.$$
If we integrate the above inequality w.r.t. $\dd\gamma_n(x',y')$ we get
$$x\cdot y+\Wc(\rho_n,\mu)\geq 0.$$
Indeed, $\int x\cdot y\,\dd\gamma_n(x',y')=x\cdot y$, $\int x'\cdot y'\,\dd\gamma_n(x',y')=\Wc(\rho_n,\mu)$, $\int x\cdot y\,\dd\gamma_n(x',y')=0$ and $\int x\cdot y\,\dd\gamma_n(x',y')=0$.
This proves $x\cdot y \geq -C$ on  $\bigcup_n \Gamma_n$ and allows to prove (3).

To prove (4), we take for instance $\mu_n:=\mu\res B(0,n)+\mu(B(0,n)^c)\delta_{v_n}$, where 
$$v_n:=\frac{1}{\mu(B(0,n)^c)}\int_{B(0,n)^c}y \,\dd\mu(y).$$ 
In this case we have $\int u^*\,\dd\tilde\mu_n\leq \int u^*\,\dd\mu$ for every convex function $ u^*$, which gives $\Wc(\rho,\mu_n)\leq \Wc(\rho,\mu)$. Combining this inequality with the semicontinuity result of (3) we get $\Wc(\rho,\mu_n)\to \Wc(\rho,\mu)$.

The same construction does not work for (5), as we want to guarantee convergence of entropies. In this case, if $\E(\rho)<+\infty$, we need to produce a sequence $\rho_n$ of absolutely continuous measures. We can take $\rho_n:=(\rho\res B(0,n))/\mu(B(0,n))$. In this case we can check explicitely that $\E(\rho_n)\to\E(\rho)$ by dominated convergence. Moreover, for every convex function $u$ we have $\int u\dd\rho_n\to\int u\dd\rho$ (first we subtract a linear part to $u$, using $\rho\in\Pu$, and then we are reduced to monotone convergence). Hence, along this sequence, the functional $\Wc(\cdot,\mu)$ is u.s.c. as an infimum of continuous functional. But the semicontinuity result of (3) provides the continuity.
\end{proof}
We are also interested in the following estimate on $\Wc$ as a function of $\rho$ in terms of the first moment of $\rho$. We first define the constant
$$c(\mu):=\frac{1}{2d}\inf\left\{\int|y\cdot e-\ell|\,\dd\mu(y)\;:\;e\in\mathbb{S}^{d-1},\ell\in\R\right\}.$$
Note that the infimum in the definition of $c(\mu)$ is actually a minimum (we minimize a function which is continuous in $e$ and $\ell$, coercive w.r.t. $\ell$, and $e$ lives in a compact set). For simplicity, we only state the estimate in the case where $\rho$ is absolutely continuous.
\begin{proposition}\label{lower bound T}
If $\int x\,\dd\rho(x)=0$ and $\rho\ll\lcal^d$, then we have $\Wc(\rho,\mu)\geq c(\mu)\int|x|\,\dd\rho(x)$. In particular, $\Wc$ satisfies an inequality of the form $\Wc(\rho,\mu)\geq c\int|x|\,\dd\rho(x)$ for $c>0$ and for every $\rho$ with $\int x\,\dd\rho(x)=0$, if and only if $\mu$ is not supported on a hyperplane.
\end{proposition}

\begin{proof}
Take $\rho$ such that $\int x\,\dd\rho(x)=0$ and $\rho\ll\lcal^d$ and select a vector $e\in \mathbb S^{d-1}$ such that $e\mapsto \int (x\cdot e)_+\,\dd\rho(x)$ is maximal. Set $A^+:=\{x:x\cdot e>0\}$ and $A^-:=\{x:x\cdot e<0\}$. By optimality conditions, this impies that $v^+:=\int_{A^+} x\,\dd\rho(x)$ is oriented as $e$. From the barycenter condition on $\rho$ and the fact that it does not charge the hyperplane $\{x\,:\,x\cdot e=0\}$, the vector  $v^-:=\int_{A^-} x\,\dd\rho(x)$ is opposite to $v^+$. Set $m^\pm:=\rho(A^\pm)$. We have $m^++m^-=1$, again from $\rho(\R^d\setminus(A^+\cup A^-))=0$.

Let $\ell$ be a value such that $\mu(\{x\,:\,x\cdot e>\ell\})\leq m^+$ and $\mu(\{x\,:\,x\cdot e<\ell\})\leq m^-$ (if $\mu$ does not give mass to the hyperplane $\{x\cdot e=\ell\}$ then both inequalities are equalities). Decompose $\mu$ into the sum of two measures $\mu^\pm$ such that
$$\spt\mu^+\subset \{x\cdot e\geq\ell\},\;\spt\mu^-\subset \{x\cdot e\leq\ell\},\;\mu^\pm(\R^d)=m^\pm.$$
Consider $\gamma$ obtained in the following way: distribute the mass of $\rho\res A^+$ onto that of $\mu^+$ via a tensor product, and do the same from $\rho\res A^-$ onto $\mu^-$, i.e.
$$\gamma=\frac{1}{m^+}(\rho\res A^+)\otimes\mu^++\frac{1}{m^-}(\rho\res A^-)\otimes\mu^-\in\Pi(\rho,\mu).$$
We have 
$$\int (x\cdot y)\dd\gamma=v^+\cdot  \frac{1}{m^+}\int y\,\dd\mu^+(y)+v^-\cdot  \frac{1}{m^-}\int y\,\dd\mu^-(y).$$
We use that $v^+$ is oriented as $e$, and $v^-=-v^+$, and we have
 $$v^+\cdot  \frac{1}{m^+}\int y\,\dd\mu^+(y)+v^-\cdot  \frac{1}{m^-}\int y\,\dd\mu^-(y)= \frac{|v^+|}{m^+}\int y\cdot e\,\dd\mu^+(y)- \frac{|v^+|}{m^-}\int y\cdot e\,\dd\mu^-(y).$$
 Since $\mu^+/m^+$ and $\mu^-/m^-$ are both probability measures, we can subtractthe same constant $\ell$ to the two integrals and get
 $$ \frac{|v^+|}{m^+}\int (y\cdot e-\ell)\,\dd\mu^+(y)- \frac{|v^+|}{m^-}\int (y\cdot e-\ell)\,\dd\mu^-(y)=|v^+|\int |y\cdot e-\ell|\,\dd(\left(\frac{\mu^+}{m^+}+(\frac{\mu^-}{m^-}\right).$$
 Then, we use $m^\pm\leq 1$ in order to estimate this last result from below with $|v^+||\int |y\cdot e-\ell|\,\dd\mu$. In order to conclude, we just need to observe that 
 $$\int |x|\,\dd\rho(x)\leq 2d \sup_{e\in \mathbb S^{d-1}} \int(x\cdot e)_+\,\dd\rho(x)=2d|v^+|.$$

The last part of the statement is easy: if $\mu$ is not concentrated on a hyperplane, then $c(\mu)>0$. If $\mu$ is concentrated on a hyperplane $H$, then take an arbitrary measure $\rho$ concentrated on the line $L$ orthogonal to $H$ and passing through the origin. One can choose it so that $\int x\,\dd\rho(x)=0$ and $\int |x|\,\dd\rho(x)>0$, while $\Wc(\rho,\mu)=0$.
\end{proof}

Finally, we also prove displacement convexity of $\Wc$ as a function of $\rho$.
\begin{proposition}\label{geodconvWc}
Let $\rho_0,\rho_1\in \Pd$ be absolutely continuous measures, and let $\rho_t=((1-t)id+tT)_\#\rho$ be the unique constant speed geodesic connecting them for the Wasserstein distance $W_2$. Then $t\mapsto \Wc(\rho_t,\mu)$ is convex on $[0,1]$. Its derivative at $t=0$ is larger than $\int (T(x)-x)\cdot y\,\dd\gamma(x)$ where $\gamma$ is an optimal transport plan from $\rho_0$ to $\mu$.
\end{proposition}

\begin{proof} Suppose for a while that $\mu\in\Pd$. Then, it is well known that $\rho\mapsto -\frac12 W_2^2(\rho,\mu)$ is a $-1$-displacement convex functional (see \cite{AmbGigSav}, Theorem 7.3.2). On the other hand, $ \rho\mapsto \frac12\int|x|^2\,\dd\rho(x)$ is $1$-convex, and the sum of the two, which gives $\Wc(\rho,\mu)-\frac 12 \int|y|^2\,\dd\mu(y)$, is displacement convex. 

To obtain the proof for general $\mu$ in $\Pu$, one needs to approximate, and part (4) of Proposition \ref{lscT} allows to do so. Hence, if $\rho_t$ is a geodesic for $W_2$, the inequality $\Wc(\rho_t,\mu_n)\leq (1-t)\Wc(\rho_0,\mu_n)+t\Wc(\rho_1,\mu_n)$ passes to the limit and implies the displacement convexity of $\rho\mapsto \Wc(\rho,\mu)$.

For the last part of the statement, we just observe that
$$\Wc(\rho_t,\mu)\geq \int (x\cdot y)\,\dd ((T_t\times id)_\#\gamma)(x,y)=\int T_t(x)\cdot y\,\dd\gamma(x,y).$$
Differentiating this last term we obtain the desired expression.
\end{proof}

\section{A variational principle for moment measures}

As we sketched in the introduction, we consider the following variational problem. We fix $\mu\in\Pu$ with $\int y\,\dd\mu(y)=0$ and not supported on a hyperplane, and we want to solve
$$(P)\qquad\min\left\{\E(\rho)+\Wc(\rho,\mu)\;:\,\rho\in\Pu\right\}.$$

\begin{theorem}
The problem (P) admits a solution, which is unique up to translation. If $\rho$ is a solution, and $ u:\R\to\R\cup\{+\infty\}$ is a convex l.s.c. function such that $\Wc(\rho,\mu)=\int u\,\dd\rho+\int u^*\,\dd\mu$ with $ u=+\infty$ on $\{\rho=0\}$, then $\rho=ce^{- u}$.
\end{theorem}

\begin{proof}
To prove existence of a solution, we take a minimizing sequence $\rho_n$. We can suppose that all $\rho_n$ have $0$ as their barycenter as translating them does not change the value of the two parts of the functional. We use the inequality
$$\E(\rho)\geq -C-\left(\int |x|\,\dd\rho(x)\right)^{1/2},$$
 and the inequality $\Wc(\rho,\mu)\geq c\int |x|\,\dd\rho(x)$ that we proved in Proposition \ref{lower bound T}. This implies that the moment $\int |x|\,\dd\rho_n(x)$ must be bounded. In particular, this gives tightness of the sequence $\rho_n$ and we assume $\rho_n\deb\rho$. Also, we know that the entropy $\E(\rho)$ is l.s.c. for the weak convergence when the first moment is bounded, and the semicontinuity of $\Wc$ along sequences with $\int x\,\dd\rho_n(x)=0$ was proven in Proposition \ref{lscT}.

This proves that a minimizer exists.

We first analyze the optimality conditions: if $\bar\rho$ is optimal and $ u$ is a convex function realizing the minimum in the definition of $\Wc(\bar\rho,\mu)$, then 
$$\rho\mapsto \int \rho\ln\rho\,\dd x+\int u\,\dd\rho$$
is minimal for $\rho=\bar\rho$. By standard convex minimization arguments this implies that $\bar\rho$ also minimizes the linearized functional
$$\rho\mapsto \int \rho(\ln\bar\rho+1)\,\dd x+\int u\,\dd\rho,$$
which implies that $\bar\rho$ is concentrated on the set of points where $\ln\bar\rho+1+ u$ is minimal. This means that $\bar\rho>0$ on every point where $ u<+\infty$, and on these points we need to have $\ln\bar\rho=C- u$, i.e. $\bar\rho=ce^{- u}$. This same formula also holds on $ u=+\infty$, since we necessarily have $\bar\rho=0$ on those points.

In particular, the optimal $\rho$ is a log-concave probability density. This implies that all its moments are finite, and we have $\rho\in\Pd$.

As for uniqueness, we suppose to have two minimizers $\rho_0,\rho_1$. We know that they must belong to $\Pd$. We use the displacement convexity of the entropy and of the $\Wc$ term and observe that the entropy is strictly convex on the geodesic $\rho_t$ unless $\rho_0$ and $\rho_1$ are obtained one from the other by translation. This gives uniqueness up to translation.
\end{proof}

\begin{remark}
We observe that (P) has no solutions if $\mu$ is concentrated on a hyperplane. Suppose this hyperplane is $\{x_d=0\}$ and take $\rho_n=\frac{1}{2^dn}\lcal^d\res (R_n)$ where $R_n=\{x=(x_1,\dots,x_d)\in\R^d\,:\,|x_i|\leq 1\;\mbox{ for }i=1,\dots,d-1, |x_d|\leq n\}$. In this case we have $\E(\rho_n)=-\ln(2n)$ and $\Wc(\rho_n,\mu)\leq \sqrt{d}\int|y|\,\dd\mu(y)$. Hence $\E(\rho_n)+\Wc(\rho_n,\mu)\to-\infty$.
\end{remark}
To prove the main result of the paper, we just need to prove that $ u$ is essentially continuous. This can be done in the following way. We recall that, given any solution $\rho$ to problem (P), we can choose as a  precise representative of $\rho$ the one given by $\rho=ce^{- u}$, with $ u$ convex and l.s.c. (hence $\rho$ is log-concave and u.s.c.).
\begin{theorem}
Let $\rho$ be the precise representative above of a solution of (P). Set $\Omega=\{ u<+\infty\}$. Then $\rho=0$ $\haus^{d-1}$-a.e. on $\partial\Omega$.
\end{theorem}

\begin{proof} Suppose on the contrary that $\rho>0$ on a set of positive $\haus^{d-1}$ measure on $\partial \Omega$. Since $\Omega$ is a convex set, we can use local coordinates and assume that this set is of the form $A=\{(x_1,x')\in\R\times\R^{d-1}\, x'\in B, x_1=h(x')\}$, where $B\subset \R^{d-1}$ has positive Lebsegue measure and $h$ is a convex real-valued function. In the same chart, $\Omega$ would be locally expressed as the set of points satisfying $x_1>h(x')$. Up to reducing the set $A$ (and hence $B$), we can suppose that $\rho$ takes values in $[a,b]$, for $0<a<b<+\infty$. We also observe that $\rho$, as it is log-concave, is locally bounded. Define $A_\ve:=\{(x_1,x')\in\R\times\R^{d-1}\, x'\in B, x_1\in[h(x'),h(x')+\ve]\}$. Since $\rho(x_1,x')$ is a continuous function of $x_1\in[h(x'),h(x')+\ve]$ [(because of its log-concavity and of the representative that we chose), we easily get $\int_{A_\ve} \rho(x)\,\dd x=O(\ve)$. 

Now, we define a new density $\rho_\ve$ as a competitor in (P). We define $T:A_\ve\to\R^d$ by $T(x_1,x')=(x_1-\ve,x')$ and we set
$$\rho_\ve=\rho\res(A_\ve^c)+\frac 12 \rho\res(A_\ve)+\frac 12 T_\#(\rho\res(A_\ve)).$$
By computing the density of $\rho_\ve$ we can check 
$$\mathcal E(\rho_\ve)=\mathcal E(\rho)-\rho(A_\ve)\ln 2.$$

In order to estimate $\Wc(\rho_\ve,\mu)$ we take the optimal function $ u$ (realizing $\Wc(\rho,\mu)=\int u\,\dd\rho+\int u^*\,\dd\mu$) and we modify it into a function $ u_\delta$ defined as follows. First, take a convex, positive and superlinear function $\chi:\R^d\to\R$ such that $\int \chi(x)\,\dd\mu(x)<+\infty$. Such a function exists because $\mu\in\Pu$. Then we fix $\delta>0$ and we take $ u=( u^*+\delta\chi)^*$. We have
$$ u_\delta(x)\leq u(x)\;\mbox{for every $x\in\Omega$}\quad  u_\delta(T(x))\leq  u(x)+\delta\chi^*\left(\frac\ve\delta\right)\;\mbox{for every $x\in A_\ve$}.$$
We have 
$$\Wc(\rho_\ve,\mu)\leq \int u_\delta\,\dd\rho_\ve+\int( u^*+\delta\chi)\,\dd\mu\leq \Wc(\rho,\mu)+\frac 12\rho(A_\ve)\delta\chi^*\left(\frac\ve\delta\right)+\delta\int\chi\,\dd\mu.$$
The optimality of $\rho$ compared to $\rho_\ve$ provides
$$\rho(A_\ve)\ln 2\leq \frac 12\rho(A_\ve)\delta\chi^*\left(\frac\ve\delta\right)+\delta\int\chi\,\dd\mu.$$
Now, use $\rho(A_\ve)\geq c_0\ve$ and choose $\delta=c\ve$ for a small constant $c$ such that $c\int\chi\,\dd\mu<\frac 12 c_0\ln2$. This gives
$$\frac 12\rho(A_\ve)\ln 2\leq \frac{c\ve}{2}\rho(A_\ve)\chi^*(c^{-1}),$$
which is impossible as $\ve\to 0$.
\end{proof}

\section{Sufficient optimality conditions}

To complete the current study, it remains to prove that every log-concave density $\rho=e^{-u}$ such that $(\nabla u)_\#\rho=\mu$ and $u$ is an essentially continuous convex function $u:\R^d\to\R\cup\{+\infty\}$ is necessariyl a minimizer of $\E(\rho)+\Wc(\rho,\mu)$. This would explain that the variational principle of the previous section finds exactly all the desired functions $u$.

As this result is not the main core of the paper, this section will be little more sketchy than the rest, and will use some results from \cite{CorKla}. Anyway, we claim that the main points of the proof are present in the paper.

The main idea, already investigated in in \cite{BlaCar,BlaMosSan} for game theory purposes, is the fact that displacement convexity is sufficient to guarantee minimality when necessary conditions are satisfied.

The result is the following.

\begin{proposition}
Suppose that $u:\R^d\to\R\cup\{+\infty\}$ is an essentially continuous convex function, consider $\bar\rho=e^{-u}$ and suppose $(\nabla u)_\#\bar\rho=\mu$. Then $\bar\rho\in\Pu$ and it solves
$$\min\left\{\E(\rho)+\Wc(\rho,\mu)\;:\,\rho\in\Pu\right\}.$$
\end{proposition}
\begin{proof} Let us consider an arbitrary $\rho\in\Pd$ with compact support and the geodesic $\rho_t=((1-t)id+tT)_\#\bar\rho$, where $T=\nabla v$ is the optimal transport from $\bar\rho$ to $\rho$. From the compact support assumption on $\rho$,we get that $T$ is bounded. From the displacement convexity of $\E$ and $\Wc(\cdot,\mu)$, in order to prove $\E(\rho)+\Wc(\rho,\mu)\geq \E(\bar\rho)+\Wc(\bar\rho,\mu)$ it is sufficient to prove
$$\frac{d}{dt}\left(\E(\rho_t)+\Wc(\rho_t,\mu)\right)_{|t=0}\geq 0.$$
The computation of the derivative is included in Propositions \ref{entropy} and \ref{geodconvWc} and we have 
$$\frac{d}{dt}\left(\E(\rho_t)+\Wc(\rho_t,\mu)\right)_{|t=0}\geq -\int (\Delta^{ac} v-d)\dd\bar\rho+\int (\nabla v(x)-x)\cdot \nabla u(x)\bar\rho(x)\dd x.$$
First we use the inequality
$$d\int e^{-u(x)}\dd x\geq \int x\cdot \nabla u(x) e^{-u(x)}\dd x$$
valid for essentially continuous convex functions $u$. This is actually an equality, but anyway the inequality we need is proven in \cite{CorKla}. Hence
$$\frac{d}{dt}\left(\E(\rho_t)+\Wc(\rho_t,\mu)\right)_{|t=0}\geq -\int (\Delta^{ac} v)\,\dd\bar\rho+\int \nabla v(x)\cdot \nabla u(x)\bar\rho(x)\dd x.$$
We then use $\Delta^{ac} v\leq \Delta v$ (as $v$ is convex and its distributional derivative is a positive measure) and we integrate by parts. We first do it on a ball $B(0,R)$:
$$-\int_{B(0,R)} (\Delta^{ac} v)\,\dd\bar\rho+\int_{B(0,R)} \nabla v(x)\cdot \nabla u(x)\bar\rho(x)\dd x=-\int_{\partial B(0,R)}\nabla v(x)\cdot n \bar\rho(x)\dd\haus^{d-1}(x).$$
We want to pass to the limit as $R\to \infty$. The secon integral may be handled using the fact that $\nabla u(x)\bar\rho(x)=-\nabla \bar\rho(x)$ and $\nabla\bar\rho\in L^1$ (this corresponds to $\mu\in\Pu$) and that $T=\nabla v\in L^\infty$.
In the first integral, we use $(\Delta v)(B(0,R))=\int_{\partial B(0,R)}\nabla v(x)\cdot n \dd\haus^{d-1}(x)\leq CR^{d-1}$, together with the exponential decay of $\bar\rho=e^{-u}$, so that $\bar\rho\in L^1(\Delta v)$. In the right-hand side, we use again the exponential decay of $\bar\rho$ with the polynomial explosion of $\haus^{d-1}(\partial B(0,R))$ and the boundedness of $\nabla v$.

Hence, we get
$$\frac{d}{dt}\left(\E(\rho_t)+\Wc(\rho_t,\mu)\right)_{|t=0}\geq\lim_{R\to\infty}-\int_{\partial B(0,R)}\nabla v(x)\cdot n \bar\rho(x)\dd\haus^{d-1}(x)=0.$$
 
 This proves the optimality of $\bar\rho$ when compared to compactly supported measures $\rho$. for a general $\rho$, we use part (5) of Proposition \ref{lscT}, which allows to approximate every $\rho\in\Pu$ with compactly supported measures $\rho_n$ so that $\E(\rho_n)+\Wc(\rho_n,\mu)\to\E(\rho)+\Wc(\rho,\mu)$.
\end{proof}

\section{From this variational problem to the one studied in \cite{CorKla}}

In this last section we want to make some short comments on (P) and connect it to the problem studied in \cite{CorKla}.


From the dual formulation of $\Wc$, we may re-write our problem as a min-min problem
$$\min\left\{\int\rho(x)\ln\rho(x)\,\dd x+\int  u(x)\rho(x)\,\dd x+\int  u^*(y)\,\dd\mu(y)\;:\;\rho\in\Pu, u\mbox{ convex}\right\}.$$
The approach of the present paper consists in considering $ u$ as a secondary variable: for every $\rho$ we compute the minimum over possible $ u$, which gives rise to the functional $\Wc(\rho,\mu)$. 

A different possible approach could consist in looking at $\rho$ as the secondary variable, i.e. considering for every $ u$ the optimal $\rho$. It is easy to see that one can compute 
$$\inf \left\{\int\rho(x)\ln\rho(x)\,\dd x+\int  u(x)\rho(x)\,\dd x\;:\;\rho\geq 0, \int\rho(x)\,\dd x=1\right\}$$
by a Lagrange multiplier approach, which means that we need to choose $\rho$ so that $\ln\rho(x)+ u=const$. Hence, $\rho$ is proportional to $e^{- u}$. Let us write $c=\int e^{- u(x)}\dd x$. We have $\rho=e^{- u}/c$ and we can compute
$$\int\rho(x)\ln\rho(x)\,\dd x+\int  u(x)\rho(x)\,\dd x=\frac 1c \int \left(e^{- u}(- u-\ln c)+ u e^{- u}\right)\dd x=-\ln c.$$
This means that we consider the functional $ u\mapsto -\ln(\int e^{- u}\dd x),$ which is a concave functional of $ u$, and we solve
$$\min\left\{\int u^*\,\dd\mu-\ln(\int e^{- u}\dd x)\;:\;  u\mbox{ convex}\right\}.$$

In the above minimization problem, the first term is convex, since $ u\mapsto u^*(y)$ is convex, but the second is concave. 

Notice that this transformation of a convex-concave minimization problem (i.e., the minimization of the difference of two convex functions) into another convex-concave minimization is what is usually known as {\it Toland duality} (see \cite{Toland1,Toland2}). In particular, we also refer to \cite{Car-toland} for the applications of this notion to the case of variational problems involving the term $-W_2^2$, and their connections to variational problems under convexity constraints. 

What \cite{CorKla} does, is to consider the same problem in terms of $ u^*$ instead of $ u$: the first term becomes linear and the second, magically, convex (thanks to a clever application of a quantitative Prekopa inequality, which orresponds, as we said in the introduction, to the displacement convexity of the entropy). One should not be astonished that they obtain a convex problem: convexity in $ u^*$ more or less corresponds to the displacement convexity of the functional that we study here in terms of $\rho$ (more precisely, convexity in $\nabla u^*$ corresponds to the convexity on generalized geodesics with base measure $\mu$).

\end{document}